\documentclass{amsart}

\usepackage{color}
\usepackage{amsfonts}
\usepackage{amssymb}
\usepackage{enumerate}
\usepackage{amsmath}
\usepackage{amsthm}
\usepackage{graphicx}
\usepackage{dsfont}
\usepackage{bbm}
\usepackage[english]{babel}

\numberwithin{equation}{section}
\allowdisplaybreaks

\newtheorem*{mt}{Main Theorem}

\newtheorem{theorem}{Theorem}[section]
\newtheorem{lemma}[theorem]{Lemma}

\newtheorem{problem}[theorem]{Problem}

\newtheorem{corollary}[theorem]{Corollary}

\newtheorem*{mmmt}{Theorem~\ref{t:Gd}}

\theoremstyle{definition}

\newtheorem{remark}[theorem]{Remark}
\newtheorem{definition}[theorem]{Definition}
\newtheorem{notation}[theorem]{Notation}
\newtheorem*{dddf}{Definition~\ref{d:gdm}}

\DeclareMathOperator{\graph}{graph}

\DeclareMathOperator{\supp}{supp}

\DeclareMathOperator{\Cov}{Cov}
\DeclareMathOperator{\Dim}{Dim}

\newcommand{\N}{\mathbb{N}}
\newcommand{\QQ}{\mathbb{Q}}
\newcommand{\Z}{\mathbb{Z}}
\newcommand{\R}{\mathbb{R}}
\renewcommand{\P}{\mathbb{P}}

\renewcommand{\P}{\mathbb{P}}
\newcommand{\E}{\mathbb{E}}
\newcommand{\eps}{\varepsilon}

\newcommand{\iA}{\mathcal{A}}
\newcommand{\iB}{\mathcal{B}}

\newcommand{\iD}{\mathcal{D}}

\newcommand{\iL}{\mathcal{L}}

\newcommand{\iF}{\mathcal{F}}
\newcommand{\iP}{\mathcal{P}}

\newcommand{\iU}{\mathcal{U}}

\begin{document}
 \thispagestyle{empty}

\title[Packing dimension of Gaussian random fields with drift]{Packing dimension of images and graphs of Gaussian random fields with drift}

\author{Rich\'ard Balka}

\address{Department of Mathematics, University of British Columbia, and Pacific Institute for the Mathematical Sciences, Vancouver, BC V6T 1Z2, Canada}

\thanks{The author was supported by the National Research, Development and Innovation Office-NKFIH, Grant 104178.}

\email{balka@math.ubc.ca}



\begin{abstract} Let $X=\{(X_1(t),\dots,X_d(t)): t\in \mathbb{R}^n\}$ be a Gaussian random field in $\R^d$ such that $X_1,\dots,X_d$ are independent, centered Gaussian random fields with continuous sample paths. Let $f\colon \mathbb{R}^n\to \mathbb{R}^d$ be a Borel map and let $A\subset \mathbb{R}^n$ be an analytic set. The main goal of the paper is to determine the almost sure value of the packing dimension of the image and graph of $X+f$ restricted to $A$ under a very mild assumption. This generalizes a result of Du, Miao, Wu and Xiao, who calculated the packing dimension of $X(A)$ if $X_1,\dots,X_d$ are independent copies of the same Gaussian random field $X_0$. Provided that $X$ is a fractional Brownian motion, our result is new even if $n=d=1$ and $f$ is continuous, and even if $f\equiv 0$ in the case of graphs.

For a fractional Brownian motion $X$ we also obtain the sharp lower bound for the packing dimension of the graph of $X$ over $A$ in terms of the Hurst index of $X$ and the packing dimension of $A$. The analogous result for images was obtained by Talagrand and Xiao. 
\end{abstract}

\maketitle

\section{Introduction} 

\subsection{The Main Theorem} \label{ss:mt} 
For a set $A\subset \R^n$ denote by $\Dim A$ and $\iP_c(A)$ the packing dimension of $A$ and the family of compactly supported Borel probability measures on $A$, respectively. For $\mu\in \iP_c(\R^n)$ and $h\colon \R^n\to \R^m$ we use the notation $\mu_h=\mu \circ h^{-1}$; if $h$ is Borel measurable on the support of $\mu$ then $\mu_h$ is a Borel probability measure on $\R^m$. For $x\in \R^n$ and $r>0$ let $B(x,r)$ denote the closed ball of radius $r$ around $x$. Let $\iB(\R^n)$ be the $\sigma$-algebra of Borel sets in $\R^n$. A random map $Z\colon \R^n\to \R^m$ defined on the probability space $(\Omega,\iF,\P)$ is called \emph{jointly measurable} if $Z\colon \R^n\times \Omega\to \R^m$ is measurable from $\iB(\R^n)\otimes \iF$ to $\iB(\R^m)$. 

In order to state our main result, we assign different notions of dimension to jointly measurable random maps. This can be considered as a generalization of the packing dimension profiles due to Falconer and Howroyd \cite{FH}. 
	
\begin{definition} \label{d:Z} Let $Z\colon \R^n\to \R^m$ be a jointly measurable random map. For all $\mu\in \iP_c(\R^n)$ define  
\[ \Dim_Z \mu=\sup\left\{\gamma: \liminf_{r\to 0+} r^{-\gamma} \E(\mu_{Z}(B(Z(t),r)))=0 \textrm{ for $\mu$-a.e.\ $t$}\right\}.\]
For every set $A\subset \R^n$ let 
\[ \Dim_{Z} A=\sup\{ \Dim_{Z} \mu: \mu\in \iP_{c}(A)\}.\]
\end{definition}

We will prove the following general theorem. 

\begin{theorem} \label{t:lbb} Let $Z\colon \R^n\to \R^m$ be a jointly measurable random map. Assume that $\mu\in \iP_c(\R^n)$ and $A\subset \R^n$ is an analytic set. Then, almost surely,  
\[\Dim \mu_Z\geq \Dim_Z \mu \quad \textrm{and} \quad \Dim Z(A)\geq \Dim_Z A. \]	
\end{theorem} 

\begin{definition} \label{d:f*}
For a map $h\colon \R^n\to \R^d$ let 
$h^{*}\colon \R^n\to \R^{n+d}$ be defined as 
\[h^{*}(t)=(t,h(t)).\]
Clearly we have $h^{*}(A)=\graph(h|_{A})$ for all $A\subset \R^n$.
\end{definition} 

\begin{definition} \label{d:reg} Let $X=\{(X_1(t),\dots,X_d(t)): t\in \R^n\}$ be a Gaussian random field in $\R^d$ such that $X_1,\dots,X_d$ are independent, centered Gaussian random fields with almost surely continuous sample paths. The \emph{canonical pseudometric} for $X_i$ is defined as 
\[\rho_i(t,s)=\sqrt{\E (X_i(t)-X_i(s))^2} \quad \textrm{for all } t,s\in \R^n.\]
We say that $X$ is \emph{regular} on $A\subset \R^n$ if there are sets $A_k$ such that $A= \bigcup_{k=1}^{\infty} A_k$ and for all $k,N\in \N^+$ there is a $c\in \R^+$ such that for all $1\leq i\leq d$ we have 
\begin{equation*}
\rho_i(t,s)\leq c\log^{-N}\left(1/|t-s|\right)\quad \textrm{for all} \quad  t,s\in A_k. 
\end{equation*} 
\end{definition}

The main result of the paper is the following.

\begin{mt} \label{t:main}  Let $X=\{(X_1(t),\dots,X_d(t)): t\in \R^n\}$ be a Gaussian random field in $\R^d$ such that $X_1,\dots,X_d$ are independent, centered Gaussian random fields with almost surely continuous sample paths. Let $f\colon \R^n\to \R^d$ be a Borel map, and let $A\subset \R^n$ be an analytic set on which $X$ is regular. If $Z=X+f$ or $Z=(X+f)^*$ then, almost surely,
\begin{equation*} \Dim Z(A)=\Dim_Z A.
\end{equation*} 	
\end{mt}

We may assume that \emph{every} sample path of $X$ is continuous. Then as a sum of two jointly measurable random maps, both $Z=X+f$ and $Z=(X+f)^{*}$ are jointly measurable. Therefore Theorem~\ref{t:lbb} yields the lower bound for $\Dim Z(A)$ in the Main Theorem. For the definition and properties of fractional Brownian motion see the Preliminaries Section. 

\begin{corollary} 
Let $\{X(t): t\in \R^n\}$ be a $d$-dimensional fractional Brownian motion and let $f\colon \R^n\to \R^d$ be a Borel map. Let $A\subset \R^n$ be an analytic set. If $Z=X+f$ or $Z=(X+f)^*$ then, almost surely,
\begin{equation*} \Dim Z(A)=\Dim_Z A.
\end{equation*} 	
\end{corollary} 

Let $\{W(t): t\in [0,\infty)^n\}$ be a Brownian sheet in $\R^d$, see \cite[Chapter~8]{Ad} for the definition. Extend the process such that $W(t)\equiv \mathbf{0}$ for all $t\in \R^n\setminus [0,\infty)^n$. Since $W$ is regular on $\R^n$ by \cite[Lemma~8.9.1]{Ad}, we obtain the following.

\begin{corollary} 
Let $\{W(t): t\in \R^n\}$ be a $d$-dimensional Brownian sheet and let $f\colon \R^n\to \R^d$ be a Borel map. Let $A\subset \R^n$ be an analytic set. If $Z=W+f$ or $Z=(W+f)^*$ then, almost surely,
\begin{equation*} \Dim Z(A)=\Dim_Z A.
\end{equation*} 	
\end{corollary}

\begin{problem}
Can we extend the Main Theorem such that 
\begin{itemize}  
\item $X_1,\dots,X_d$ are not necessarily independent, 
\item  $Z(t)=f(t,X(t))$ with a Borel map $f\colon \R^{n+d}\to \R^m$?
\end{itemize} 
\end{problem}

\subsection{Related work and further results} \label{ss:fr} 

In this subsection we discuss the related results in detail, including the case of Hausdorff dimension as well. Particular emphasis will be given to fractional Brownian motion. Our sharp lower bound of the packing dimension of the graph of fractional Brownian motion follows from Theorems~\ref{t:ad} and \ref{t:ad2}.   

\subsubsection{The case of Hausdorff dimension}
Let $\dim A$ denote the Hausdorff dimension of a set $A\subset  \R^n$, for its definition and properties see Falconer~\cite{F} or Mattila~\cite{Ma}. The following theorem is due to Kahane~\cite[Chapter~18]{Kh}.

\begin{theorem}[Kahane] \label{t:Kh}
Let $\{X(t): t\in \R^n\}$ be a $d$-dimensional fractional Brownian motion of Hurst index $\alpha\in (0,1)$. For a Borel set $A\subset \R^n$, almost surely, 
\begin{align*}
\dim X(A) &=\min\{(1/\alpha)\dim A, d \}, \\
\dim X^{*}(A)&=\min \{ (1/\alpha)\dim A, \dim A+d(1-\alpha)\} 
\end{align*}  
\end{theorem}
  
If $n\leq \alpha d$ then Monrad and Pitt~\cite{MP} proved the next uniform dimension result.

\begin{theorem}[Monrad--Pitt] \label{t:MP} Let $0<\alpha<1$ and let $n,d\in \N^+$ with $n\leq \alpha d$. Let $\{X(t): t\in \R^n\}$ be a $d$-dimensional fractional Brownian motion of Hurst index $\alpha$. Then, almost surely, for every Borel set $A\subset \R^n$ we have
\begin{equation} \label{eq:ud} \dim X(A)=\dim X^{*}(A)=(1/\alpha) \dim A. 
\end{equation} 
\end{theorem} 
Monrad and Pitt only proved $\dim X (A)=(1/\alpha) \dim A$, but the other equation easily follows. Indeed, since Hausdorff dimension cannot increase under projections, $\dim X(A)\leq \dim X^{*}(A)$ holds. As $X^{*}$ is almost surely locally $\gamma$-H\"older continuous for all $\gamma<\alpha$, we also have $\dim X^{*}(A)\leq (1/\alpha) \dim A$. 

In the case $n>\alpha d$ Monrad and Pitt~\cite{MP} proved that a $d$-dimensional fractional Brownian motion $\{X(t): t\in \R^n\}$ satisfies
$\dim X^{-1}(\mathbf{0})=n-\alpha d$, witnessing that none of the equations of \eqref{eq:ud} holds uniformly. 
 
The case of fractional Brownian motion with drift was recently solved by Peres and Sousi~\cite{PS}. In order to state their result we need the following definition. 

\begin{definition} 
Let $A\subset \R\times \R^d$ and $\alpha\in [0,1]$. For all 
$\beta>0$ the \emph{$\alpha$-parabolic $\beta$-dimensional Hausdorff content} of $A$ is defined as
\[\Psi^{\beta}_{\alpha}(A)=\inf\left\{\sum_j \delta^{\beta}_j: A\subset 
[a_j, a_j+\delta_j]\times [b_{j,1},b_{j,1}+\delta_j^{\alpha}] 
\times \dots \times   [b_{j,d},b_{j,d}+\delta_j^{\alpha}]\right\},\]
where the infimum is taken over all countable covers of $A$ by rectangles of the form given above. The \emph{$\alpha$-parabolic Hausdorff dimension of $A$} 
is defined by 
\[\dim_{\Psi,\alpha}(A)=\inf\{\beta: \Psi^{\beta}_{\alpha}(A)=0\}. \]
\end{definition} 

The above definition was introduced by Kaufman~\cite{Ku} in the case $\alpha=1/2$ to study the Hausdorff dimension of $B^{-1}(F)\cap E$, where $B\colon [0,\infty)\to \R$ is a standard Brownian motion, and $F\subset \R$ and $E\subset [0,\infty)$ are fixed compact sets.

\begin{theorem}[Peres--Sousi] Let $\{X(t): t\in [0,1]\}$ be a $d$-dimensional fractional Brownian motion of Hurst index $\alpha\in (0,1)$. Let $f\colon \R\to \R^d$ be a Borel function and let $A\subset [0,1]$ be a Borel set. If $\beta=\dim_{\Psi,\alpha} \graph(f|_{A})$, then almost surely
\begin{align*} 
\dim \, (X+f)(A)&=\min\{\beta/\alpha, d\}, \\
\dim \, (X+f)^*(A) &=\min\{\beta/\alpha, \beta+d(1-\alpha)\}.
\end{align*} 	
 \end{theorem}
  
\subsubsection{The case of packing dimension}  
The following uniform dimension result is due to Xiao~\cite{X2}, recall Theorem~\ref{t:MP} and the discussion thereafter. 

\begin{theorem}[Xiao] Let $0<\alpha<1$ and let $n,d\in \N^+$ with $n\leq \alpha d$. Let $\{X(t): t\in \R^n\}$ be a $d$-dimensional fractional Brownian motion of Hurst index $\alpha$. Then, almost surely, for every Borel set $A\subset \R^n$ we have
\begin{equation} \label{eq:ud2} \Dim X (A)=\Dim X^{*}(A)= (1/\alpha) \Dim A. 
\end{equation} 
\end{theorem} 

From now on assume that $n>\alpha d$. Xiao~\cite{X2} showed that $\Dim X^{-1}(\mathbf{0})=n-\alpha d$, witnessing that none of the equations of \eqref{eq:ud2} holds uniformly.
It was a long-standing open problem whether for every fixed Borel set $A\subset \R^n$ we have
\begin{equation} \label{eq:dbl} 
\Dim X(A)=\min\left\{d, (1/\alpha) \Dim A\right\}.
\end{equation}  
The upper bound of $\Dim X(A)$ coming from \eqref{eq:dbl} follows from the 
local H\"older continuity of $X$.
Since Hausdorff dimension is less than equal to packing dimension (see Tricot~\cite{T}), Theorem~\ref{t:Kh} easily implies that \eqref{eq:dbl} holds for all Borel sets $A\subset \R^n$ satisfying $\dim A=\Dim A$. However, Talagrand and Xiao~\cite{TX} proved that \eqref{eq:dbl} does not hold in general. For the sake of technical simplicity the following theorems discuss only the case $n=1$. The next one is \cite[Theorem~4.1]{TX}. 

\begin{theorem}[Talagrand--Xiao] \label{t:TX}
Let $d\in \N^+$ and let $0<\alpha<1$ with $1>\alpha d$. Let $\{X(t):t\in \R\}$ be a $d$-dimensional fractional Brownian motion of Hurst index $\alpha$. Let $A\subset \R$ be a compact set with $\Dim A=\beta$. 
Then, almost surely, we have 
\[\Dim X(A)\geq \frac{\beta d}{\alpha d+\beta(1-\alpha d)}.\]
\end{theorem}

For the following theorem see \cite[Proposition~3.1]{TX}.

\begin{theorem}[Talagrand--Xiao] \label{t:TX2}
For every $\beta \in (0,1)$ there exists a compact set 
$A_{\beta}\subset [0,1]$ such that $\Dim A_{\beta}=\beta$ with the following property. If $f\colon A_{\beta}\to \R^d$ is an $\alpha$-H\"older continuous function with some $d\in \N^+$ and $\alpha\in (0,1)$, then we have 
\[\Dim f(A_{\beta})\leq \frac{\beta d}{\alpha d+\beta(1-\alpha d)}.\]
\end{theorem}

Since $X\colon \R\to \R^d$ is almost surely $\gamma$-H\"older continuous on $[0,1]$ for all $\gamma \in (0,\alpha)$, the above results imply that, almost surely, 
\[\Dim X(A_{\beta})= \frac{\beta d}{\alpha d+\beta(1-\alpha d)}<\min\{\beta/\alpha,d\}.\]
Thus the sets $A_{\beta}$ witness that Theorem~\ref{t:TX} is sharp and \eqref{eq:dbl} does not hold in general.
 
For graphs it is natural to ask whether for all Borel sets $A\subset \R^n$ we have
\begin{equation} \label{eq:dbl2} 
\Dim X^{*}(A)=\min\left\{(1/\alpha) \Dim A, \Dim A+d(1-\alpha)\right\}.
\end{equation}  
The upper bound of
$\Dim X^{*}(A)$ coming from \eqref{eq:dbl2} follows from \cite[Lemma~2.2]{X0}. If $A\subset \R^n$ is a Borel set with $\dim A=\Dim A$ then \eqref{eq:dbl2} holds, see the arguments after \eqref{eq:dbl}. We will prove the following.

\begin{theorem} \label{t:ad} Let $d\in \N^+$ and let $0<\alpha<1$ with $1>\alpha d$. Let $\{X(t):t\in \R\}$ be a $d$-dimensional fractional Brownian motion of Hurst index $\alpha$. Let $A\subset \R$ be an analytic set with $\Dim A=\beta$. 
Then, almost surely, we have 
\[\Dim X^{*}(A)\geq  \max\left\{ \frac{\beta d}{\alpha d+\beta(1-\alpha d)},\beta(d+1-\alpha d)\right\}.\]
\end{theorem} 

We will also show the following analogue of Theorem~\ref{t:TX2}. 

\begin{theorem}\label{t:ad2}
For every $\beta \in (0,1)$ there exists a compact set 
$A_{\beta}\subset [0,1]$ such that $\Dim A_{\beta}=\beta$ with the following property. If $f\colon A_{\beta}\to \R^d$ is an $\alpha$-H\"older continuous function with some $d\in \N^+$ and $\alpha\in (0,1)$, then we have
\[\Dim f^{*}(A_{\beta})\leq \max\left\{ \frac{\beta d}{\alpha d+\beta(1-\alpha d)},\beta(d+1-\alpha d)\right\}.\]
\end{theorem}

Similarly as above, this implies that, almost surely, 
\[\Dim X^{*}(A_{\beta})= \max\left\{ \frac{\beta d}{\alpha d+\beta(1-\alpha d)},\beta(d+1-\alpha d)\right\}<\min\{\beta/\alpha, \beta+d(1-\alpha)\}.\]
Thus the sets $A_{\beta}$ witness that Theorem~\ref{t:ad} is sharp and \eqref{eq:dbl2} does not hold in general.

Xiao~\cite{X} determined the almost sure packing dimension of 
$X(A)$. Before stating his result we need some preparation.

\begin{notation} \label{n:Fd}
Let $\mu$ be a Borel probability measure on $\R^{d}$. For all 
$\beta>0$, $x\in \R^d$, and $r>0$ define 
\[F_{\beta}^{\mu}(x,r)=\int_{\R^d} \min\{1, r^{\beta} |y-x|^{-\beta}\} \, \mathrm{d} \mu(y). \]
\end{notation} 

The following definition is due to Falconer and Howroyd \cite{FH}. 

\begin{definition} \label{d:FH}
For a Borel probability measure $\mu$ on $\R^d$ and 
$\beta>0$ define the \emph{$\beta$-dimensional packing dimension profile of $\mu$} as 
\[\Dim_\beta \mu=\sup\left\{\gamma: \liminf_{r\to 0+} r^{-\gamma} F_{\beta}^{\mu}(x,r)=0 \textrm{ for $\mu$-a.e.}\ x\in \R^d \right\}.\]
For $A\subset \R^d$ we define the \emph{$\beta$-dimensional packing dimension profile of $A$} as 
\[ \Dim_\beta A=\sup\{\Dim_\beta \mu: \mu \in \iP_c(A)\}.\] 
\end{definition} 

\begin{theorem}[Xiao] \label{t:X} Let $\{X(t): t\in \R^n\}$ be a $d$-dimensional fractional Brownian motion of Hurst index $\alpha\in (0,1)$. For each analytic set $A\subset \R^n$, almost surely, 
\[ \Dim X(A)=(1/\alpha) \Dim_{\alpha d} A. \]  
\end{theorem} 

Xiao~\cite{X} also claims that if $\{W(t): t\in \R^n\}$ is a $d$-dimensional Brownian sheet and $A\subset (0,\infty)^n$ is an analytic set, then $\Dim W(A)=2\Dim_{d/2} A$ almost surely. Khoshnevisan, Schilling, and Xiao \cite{KSX} 
determined $\Dim X(A)$ for every L\'evy process
$\{X(t): t\in [0,\infty)\}$ in $\R^d$. Zhang~\cite{Z} generalized this for all additive L\'evy processes $\{X(t): t\in [0,\infty)^n\}$ in $\R^d$. Let $\{B(t): t\in [0,\infty)\}$ be a standard $d$-dimensional Brownian motion, then $\{B^{*}(t): t\in [0,\infty)\}$ is a L\'evy process in $\R^d$, thus \cite{KSX} determines the almost sure value of $\Dim B^*(A)$. The author could not find an analogue of Theorem~\ref{t:X} for graphs if $\alpha\neq 1/2$ or $n>1$. Du, Miao, Wu, and Xiao~\cite{DMWX} generalized Theorem~\ref{t:X} (assuming a slightly stronger version of regularity) as follows. 
\begin{theorem}[Du--Miao--Wu--Xiao] \label{t:DMWX}
Let $X=\{(X_1(t),\dots,X_d(t)): t\in \R^n\}$ be a Gaussian random field in $\R^d$ with almost surely continuous sample paths such that $X_1,\dots,X_d$ are independent copies of a centered Gaussian random field $X_0$. Let $A\subset \R^n$ be an analytic set on which $X$ is regular. Then, almost surely,
 \begin{equation*} \Dim X(A)=\Dim_X A.
 \end{equation*} 		
\end{theorem} 
Note that $\Dim_X A$ was given in \cite{DMWX} as a direct generalization of the packing dimension profiles, which is not enough for us to settle the Main Theorem. For further references see also \cite{DMWX}.

Loosely speaking, Charmoy, Peres, and Sousi~\cite[Proposition~2.1]{CPS} calculated the almost sure upper Minkowski dimension of $(X+f)(A)$, where $\{X(t): t\in [0,1]\}$ is a L\'evy process in $\R^d$, $f\colon [0,1]\to \R^d$ is a c\`adl\`ag function, and $A\subset [0,1]$ is an arbitrary fixed set. In particular, \cite[Theorems~1.2,~1.3]{CPS} and Baire's category theorem imply the following.

\begin{theorem}[Charmoy--Peres--Sousi] \label{t:CPS} Suppose that $\{B(t): t\in [0,1]\}$ is a standard $d$-dimensional Brownian motion, $f\colon [0,1]\to \R^d$ is a c\`adl\`ag function, and $A\subset [0,1]$ is a closed set. Then there exist constants $c_1,c_2$ such that, almost surely, 
\[\Dim \, (B+f)(A)=c_1 \quad \textrm{and} \quad \Dim \, (B+f)^{*}(A)=c_2.\]
\end{theorem} 

\subsection{Overview}
Main Theorem is a far-reaching generalization of Theorems~\ref{t:DMWX} and~\ref{t:CPS}. Compared to Theorem~\ref{t:DMWX} the new features are the following.
\begin{enumerate}[(1)]
\item \label{o:1} We do not need to assume that the fields $X_1,\dots,X_d$ are identical copies. 
\item \label{o:2}We can cover the case of graphs as well. 
\item \label{o:3} We can add an arbitrary Borel drift to $X$. 
\end{enumerate} 
Falconer and Howroyd proved a new formula for the packing dimension of measures in Euclidean spaces, see Theorem~\ref{t:FH}. The novelty which delivers \eqref{o:1} and \eqref{o:2} is Theorem~\ref{t:Gd}, which generalizes Theorem~\ref{t:FH} in two different directions. 
On the one hand, Theorem~\ref{t:Gd} uses the larger kernel $I_d(x)=\prod_{i=1}^{d} \min\{1,|x_i|^{-1}\}$ instead of $\min\{1,|x|^{-d}\}$ for $x\in \R^d$. This allows us to deal with the coordinate fields $X_i$ separately, leading to \eqref{o:1}. 
On the other hand, Theorem~\ref{t:Gd} naturally handles $\R^{n+d}$ as the product space $\R^n\times \R^d$. This makes it easier to work with graphs of functions from $\R^n$ to $\R^d$, allowing us to fulfill \eqref{o:2}.
In order to achieve \eqref{o:3} we make use of Definition~\ref{d:Z} and carry out the analysis of Lemma~\ref{l:gf}. Charmoy, Peres, and Sousi~\cite{CPS} emphasized that ``the presence of the drift $f$ implies that we cannot use techniques relying on self-similarity of the paths''. 

\subsection{Organization}
In Section~\ref{s:pr} we recall and introduce some notation and definitions, we devote a separate subsection to some properties of the packing dimension. In Section~\ref{s:main} we prove the Main Theorem. In Subsection~\ref{ss:lb} we prove Theorem~\ref{t:lbb} and verify that Definition~\ref{d:Z} makes sense and is independent of the norm on $\R^m$. In Subsection~\ref{ss:ub} we show that the upper bound for $\Dim Z(A)$ holds in the Main Theorem. The equivalent definitions of Corollary~\ref{c:1} exhibit the connection between Definition~\ref{d:Z} and the packing dimension profiles. Although Theorem~\ref{t:Gd} is used for the proof of the Main Theorem, we postpone its proof to Section~\ref{s:tech} because of its technical nature. In Section~\ref{s:tech} we use a number of methods developed in \cite{FM,FH}. Section~\ref{s:graph} is dedicated to the proofs of Theorems~\ref{t:ad} and \ref{t:ad2}, which are based on a paper of Talagrand and Xiao~\cite{TX}.

\section{Preliminaries} \label{s:pr} 

Let $\mathbf{0}$ denote the origin of $\R^d$. Unless stated otherwise, we endow $\R^d$ with the Euclidean norm $|\cdot|$ and $||\cdot||$ will denote the maximum norm  defined by $||x||=\max_{1\leq i\leq d} |x_i|$. Recall that the closed ball of radius $r$ around $x$ in the Euclidean norm is denoted by $B(x,r)$, and let $D(x,r)$ denote the corresponding ball in the maximum norm. We use the convention $\R^{0}=\{0\}$. Let $\iP(\R^d)$ be the set of Borel probability measures on $\R^d$. The \emph{support} of $\mu\in \iP(\R^d)$ is denoted by $\supp(\mu)$, which is the minimal closed set $F$ for which $\mu(\R^d\setminus F)=0$. Recall that for $A\subset \R^d$ we have
\[ \iP_c(A)=\{\mu\in \iP(\R^d): \supp(\mu)\subset A \textrm{ is compact}\}.\] 
A set $A\subset \R^d$ is called \emph{analytic} or \emph{Souslin} if it is a continuous image of a complete, separable metric space. By \cite[Proposition~14.2]{Ke} every Borel set in $\R^d$ is analytic. Let $N=N(0,1)$ denote a standard normal random variable with density function $\varphi(x)=(2\pi)^{-1/2} \exp(-x^2/2)$. Let $\mathbbm{1}_{A}$ be the characteristic function of the set $A$. A map $f\colon \R^n\to \R^d$ is called \emph{locally $\gamma$-H\"older continuous} if for each compact set 
$K\subset \R^n$ there is a finite constant $c=c(K)$ such that $|f(x)-f(y)|\leq c|x-y|^{\gamma}$ for all $x,y\in K$. We use the notation $a\lesssim_{d_1,\dots,d_k} b$ if there is a finite constant $c=c(d_1,\dots,d_k)$ such that $a\leq cb$. We write $a\lesssim b$ if $a\leq cb$ with an absolute $c$.

\begin{definition}
Let $0<\alpha<1$. Then $X=\{(X_1(t),\dots,X_d(t)): t\in \R^n\}$ is called a 
\emph{$d$-dimensional fractional Brownian motion of Hurst index $\alpha$} if $X_1,\dots,X_d$ are i.i.d.\ copies of a centered, real-valued Gaussian random field $\{X_0(t): t\in \R^n\}$ such that $X_0(\textbf{0})=0$, almost every sample path of $X_0$ is continuous, and for all $t,s\in \R^n$ we have 
\[
\Cov (X_0(t),X_0(s))=(1/2)(|t|^{2\alpha}+|s|^{2\alpha}-|t-s|^{2\alpha}).
\]
\end{definition}
For each $t,s\in \R^n$ the variable $X_0(t)-X_0(s)$ has distribution $|t-s|^{\alpha}N$, and $X$ is almost surely locally $\gamma$-H\"older continuous for
all $\gamma<\alpha$. For more information on
fractional Brownian motion see \cite[Chapter~8]{Ad} and \cite[Chapter~18]{Kh}.

\subsection{Packing dimension and some of its properties} \label{ss:pack}
Let $(E,\rho)$ be a pseudometric space. If $E$ is totally bounded then for all $r>0$ let $N_{\rho}(E,r)=N(E, r)$ be the smallest number of closed balls of radius $r$ whose union covers $E$. The \emph{upper Minkowski dimension} of $E$ is defined as
\[
\overline{\dim}_M E=\limsup_{r\to 0+} \frac{\log N(E,r)}{\log (1/r)}.
\]
Let $\overline{\dim}_M \emptyset=0$ and let $\overline{\dim}_M E=\infty$ if $E$ is not totally bounded. The \emph{packing dimension} of $E$ is defined as
\[\Dim E=\inf\left\{\sup_{i}\overline{\dim}_M E_i: E=\bigcup_{i=1}^{\infty} E_i\right\}.\] 
The \emph{packing dimension} of a Borel probability measure $\mu$ on $\R^d$ is defined as
\begin{equation*} \Dim \mu=\inf\{ \Dim E: E\subset \R^d \textrm{ is Borel and } \mu(E)>0\}.
\end{equation*}

The following theorem is due to Hu and Taylor~\cite[Lemma~4.1]{HT}.

\begin{theorem}[Hu--Taylor] \label{t:pmu}
Let $\mu$ be a Borel probability measure on $\R^d$. Then
\[\Dim \mu=\sup\left\{\gamma: \liminf_{r\to 0+} r^{-\gamma} \mu(B(x,r))=0 \textrm{ for $\mu$-a.e.\ $x\in \R^d$}\right\}.\]	
\end{theorem} 

The following result is \cite[Lemma~4.2]{SX}, which is based on \cite{L}.

\begin{lemma} \label{l:h} 
Let $h\colon \R^n\to \R^d$ be a Borel map and let $A\subset \R^n$ be an analytic set. Then 
\[\Dim h(A)=\sup\{ \Dim \mu_h: \mu\in \iP_c(A)\}.\]	
\end{lemma} 

\begin{corollary} \label{c:h2} Let $A\subset \R^n$ be an analytic set. Then 
\[\Dim A=\sup\{ \Dim K: K\subset  A \textrm{ is compact}\}.\]
\end{corollary}

Recall Notation~\ref{n:Fd} in the case $\beta=d$. For all $\mu\in \iP(\R^d)$, $x\in \R^d$, and $r>0$ let
\[F_{d}^{\mu}(x,r)=\int_{\R^d} \min\{1, r^d |y-x|^{-d}\} \, \mathrm{d} \mu(y). \]
The following theorem is due to Falconer and Howroyd \cite[Corollary~3]{FH}. 

\begin{theorem}[Falconer--Howroyd] \label{t:FH} Let $\mu\in \iP(\R^d)$. Then
$$\Dim \mu=\sup\left\{ \gamma: \liminf_{r\to 0+} r^{-\gamma} F_{d}^{\mu}(x,r)=0 \textrm{ for $\mu$-a.e.}\ x\in \R^d\right\}.$$
\end{theorem}

\begin{notation} 
For all $d\in \N^+$ and $x=(x_1,\dots,x_d)\in \R^d$ let 
\[ I_d(x)=\prod_{i=1}^d \min\{1,|x_i|^{-1}\}\]
with the convention $\min\{1,0^{-1}\}=1$. 
\end{notation}

\begin{definition} \label{d:gdm}
Let $n\in \N$ and let $d\in \N^+$.
Let $\mu$ be a Borel probability measure on $\R^{n+d}$. For $x=(u,v)\in \R^n\times \R^d$ and $r>0$ let $\mu_{u,r}$ denote the Borel measure on $\R^d$ such that for every Borel set $E\subset \R^d$ we have
\begin{equation*} \mu_{u,r}(E)=\mu(D(u,r)\times E).
\end{equation*}
Let us define
\begin{equation*}
G_{d}^{\mu}(x,r)=\int_{\R^d} I_d\left( \frac{y-v}{r}\right) \, \mathrm{d} \mu_{u,r}(y).
\end{equation*}
\end{definition} 

We will prove the following result. 

\begin{theorem}  \label{t:Gd}
Let $\mu\in \iP(\R^{n+d})$ for some $n\in \N$ and $d\in \N^+$. Then
\[ \Dim \mu=\sup\left\{\gamma: \liminf_{r\to 0+} r^{-\gamma} G_{d}^{\mu}(x,r)=0 \textrm{ for $\mu$-a.e.}\ x\in \R^{n+d}\right\}.
\]
\end{theorem}

Note that for all $\mu\in \iP(R^d)$, $x\in \R^d$ and $r>0$ we have
\[\mu(B(x,r))\leq F_d^{\mu}(x,r)\leq G_{d}^{\mu}(x,r),\]
so provided Theorem~\ref{t:pmu}, Theorem~\ref{t:Gd} generalizes Theorem~\ref{t:FH} even in the case $n=0$. Theorem~\ref{t:Gd} will allow us to apply our methods for more general Gaussian fields, and extend it from images to graphs as well. Its proof is deferred to Section~\ref{s:tech}.

\begin{definition}
For each map $f\colon \R^n\to \R^d$, for all 
$x,y\in \R^n$ and $r>0$ define 
\[I^{f}_d(x,y,r)=
I_{d}\left(\frac{f(y)-f(x)}{r}\right).\]
\end{definition}

 Theorem~\ref{t:Gd} implies the following.

\begin{corollary} \label{c:gdf}
Let $f\colon \R^n\to \R^d$ be a Borel map and let 
$\mu\in \iP(\R^n)$. Then
\begin{align*} 
\Dim \mu_f&=\sup\left\{\gamma: \liminf_{r\to 0+} r^{-\gamma} \int_{\R^n} I^f_{d}(x,y,r) \, \mathrm{d} \mu(y)=0 \textrm{ for $\mu$-a.e.}\ x\in \R^n \right\}, \\
\Dim \mu_{f^*}&=\sup\left\{\gamma: \liminf_{r\to 0+} r^{-\gamma} \int_{D(x,r)} I^f_{d}(x,y,r) \, \mathrm{d} \mu(y)=0 \textrm{ for $\mu$-a.e.}\ x\in \R^n \right\}.
\end{align*}
\end{corollary}

\section{The proof of the Main Theorem} \label{s:main} 

\subsection{The lower bound} \label{ss:lb}  

First we prove in Lemma~\ref{l:1} that the expectation in Definition~\ref{d:Z} exists and admits a useful integral representation. The main goal of this subsection is to prove Theorem~\ref{t:lbb}. First we need a technical lemma. 

\begin{lemma} \label{l:G}
Let $Z\colon \R^n\to \R^m$ be a jointly measurable random map defined on the probability space $(\Omega,\iF, \P)$. Then for each Borel set $B\subset \R^m$ the set
\begin{equation} \label{eq:GB}
\Gamma(B)=\{(\omega,t,s)\in \Omega \times \R^n \times \R^n: 
Z(s,\omega)-Z(t,\omega)\in B\} 
	\end{equation} 
is in $\iF \otimes \iB(\R^n)\otimes \iB(\R^n)$. 
\end{lemma} 

\begin{proof}
As $\{B\subset \R^n: \Gamma(B)\in \iF \otimes \iB(\R^n)\otimes \iB(\R^n)\}$ is clearly a $\sigma$-algebra and $\Gamma(\emptyset)=\emptyset$, it is enough to prove the statement for an arbitrarily fixed non-empty open set $U\subset \R^m$. Let $\iU$ be a countable open basis of $\R^m$. Then there are sequences $\{V_i\}_{i\geq 1}$ and $\{W_i\}_{i\geq 1}$ such that $V_i,W_i\in \iU$ for all $i$ and 
\[\Gamma(U)=\bigcup_{i=1}^{\infty}\{(\omega,t,s): Z(t,\omega)\in V_i \textrm{ and } Z(s,\omega)\in W_i\}. \]
As $Z$ is jointly measurable, the sets 
\[E_i=\{(\omega,t): Z(t,\omega)\in V_i\} \quad \textrm{and} \quad  F_i=\{(\omega,s): Z(s,\omega)\in W_i\} \] 
satisfy $E_i,F_i\in \iF \otimes \iB(\R^n)$ for all $i$. Let $p,q\colon \Omega\times \R^n\times \R^n\to \Omega \times \R^n$ be the projections defined as 
\[p(\omega,t,s)=(\omega,t) \quad \textrm{and} \quad q(\omega,t,s)=(\omega,s).\]
Then clearly for all $i$ we have $p^{-1}(E_i), 
q^{-1}(F_i)\in \iF \otimes \iB(\R^n)\otimes \iB(\R^n)$, so 
\[\Gamma(U)=\bigcup_{i=1}^{\infty} \left(p^{-1}(E_i)\cap q^{-1}(F_i)\right)\in \iF \otimes \iB(\R^n)\otimes \iB(\R^n).\]
The proof is complete.  
\end{proof}

\begin{notation} 
For a set $\Gamma\subset \Omega\times \R^n\times \R^n$ and for $\omega\in \Omega$ and $t\in \R^n$ let 
\begin{align*} \Gamma_t&=\{(\omega,s)\in \Omega\times \R^n: (\omega,t,s)\in \Gamma\}, \\
\Gamma_{\omega,t}&=\{s\in \R^n: (\omega,t,s)\in \Gamma\}.
\end{align*}
\end{notation}

Below $||\cdot||$ and $B(z,r)$ denote a general norm and a corresponding ball. 

\begin{lemma} \label{l:1}
Let $Z\colon \R^n\to \R^m$ be a jointly measurable random map. Endow $\R^m$ with a norm $||\cdot||$. If $\mu\in \iP(\R^n)$ then for all $t\in \R^n$ and $r>0$ we have  
\[\E(\mu_Z(B(Z(t),r)))=\int_{\R^n} \P(||Z(s)-Z(t)||\leq r) \, \mathrm{d}  \mu(s).\]
\end{lemma}

\begin{proof} 
Assume that $Z$ is defined on the probability space $(\Omega, \iF, \P)$. Fix $t\in \R^n$ and $r>0$ arbitrarily. Then $B=B(\mathbf{0},r)$ is compact. Define $\Gamma=\Gamma(B)$ according to \eqref{eq:GB}, and let
\[ \iA=\{(\omega,s)\in \Omega \times \R^n: Z(s,\omega)-Z(t,\omega)\in B\}.\]
As $\Gamma\in  \iF \otimes \iB(\R^n)\otimes \iB(\R^n)$ by Lemma~\ref{l:G} and $\iA=\Gamma_t$, we obtain $\iA\in \iF \otimes \iB(\R^n)$. Thus $\iA$ is measurable with respect to $\P\times \mu$, and Fubini's theorem implies that 
\[\E(\mu_Z(B(Z(t),r)))=\int_{\Omega\times \R^n} \mathds{1}_{\iA} \, \mathrm{d}  (\P\times \mu)(\omega,s)=\int_{\R^n} \P(||Z(s)-Z(t)||\leq r) \, \mathrm{d}  \mu(s),\]
which finishes the proof.
\end{proof}

\begin{remark} \label{r:1} 
Lemma~\ref{l:1} shows that Definition~\ref{d:Z} is correct with any norm $||\cdot ||$ on $\R^m$. As all norms on $\R^m$ are equivalent, it easily follows that the definition is independent of the norm.
\end{remark}

Before proving Theorem~\ref{t:lbb} we need the following technical lemma. 

\begin{lemma} \label{l:pr} Let 
$Z\colon \R^n\to \R^m$ be a jointly measurable random map defined on the probability space $(\Omega,\iF, \P)$. Let $\mu\in \iP(\R^n)$ and let $\gamma\in \R$. Then the set
\[\iD=\left\{(\omega,t)\in \Omega \times \R^n: \liminf_{r\to 0+} r^{-\gamma} \mu(\{s: Z(s,\omega) \in B(Z(t,\omega),r)\})=0\right\}\]
is in $\iF \otimes \iB(\R^n)$.
\end{lemma}

\begin{proof} Lemma~\ref{l:1} yields that $\{\mu_Z (B(Z(t),r))\}_{r>0}$ is a set of random variables which are clearly non-decreasing in $r$. Therefore for all $\eps>0$ the random set $\{r>0: r^{-\gamma} \mu_Z (B(Z(t),r))<\eps\}$ is always a union of non-degenerate intervals. Thus
\[ \iD=
\bigcap_{k=1}^{\infty} \left(\bigcup_{r\in \QQ\cap (0,1/k)} \iD_{k,r}\right),\]
where
\[\iD_{k,r}=\left\{(\omega,t)\in  \Omega \times \R^n: 
r^{-\gamma} \mu(\{s: Z(s,\omega) \in B(Z(t,\omega),r)\})<1/k\right\}.\]
Fix $k\in \N^+$ and $r\in \QQ\cap (0,1/k)$, it is enough to show that $\iD_{k,r}\in \iF\otimes \iB(\R^n)$. Let $F\colon \Omega\times \R^n\to [0,1]$ be defined as 
\[F(\omega,t)=\mu(\{s: Z(s,\omega) \in B(Z(t,\omega),r)\}),\]
it is clearly enough to prove that $F$ is measurable for $\iF \otimes \iB(\R^n)$. Let $B=B(\mathbf{0},r)$ and let 
\[\Gamma=\{(\omega,t,s)\in \Omega \times \R^n \times \R^n: Z(s,\omega)-Z(t,\omega)\in B\}.\]
Lemma~\ref{l:G} yields that $\Gamma\in \iF \otimes \iB(\R^n)\otimes \iB(\R^n)$, and we have $F(\omega,t)=\mu(\Gamma_{\omega,t})$. Therefore \cite[Theorem~17.25]{Ke} implies that $F$ is measurable for $\iF\otimes \iB(\R^n)$, see also \cite[Exercise~17.36]{Ke}. The proof is complete.
\end{proof}

\begin{proof}[Proof of Theorem~\ref{t:lbb}] 
Let $Z\colon \R^n\to \R^m$ be a jointly measurable random map defined on the probability space $(\Omega,\iF,\P)$. 
Assume that $\mu\in \iP_c(\R^n)$ and $A\subset \R^n$ is an analytic set. 
First we prove that, almost surely, 
\begin{equation} \label{eq:muz} 
\Dim \mu_Z\geq \Dim_Z \mu.
\end{equation} 
Fix an arbitrary $\gamma<\Dim_{Z} \mu$, it is enough to show that, almost surely, $\Dim \mu_Z \geq \gamma$. The definition of $\Dim_Z \mu$ yields that for $\mu$ almost every $t \in \R^n$ we have
\begin{equation} \label{eq:limi} 
\liminf_{r\to 0+} r^{-\gamma} \E(\mu_Z(B(Z(t),r)))=0.
\end{equation} 
Fix an $t\in \R^n$ for which \eqref{eq:limi} holds. 
The argument of the beginning of the proof of Lemma~\ref{l:1} yields that  
\[\liminf_{r\to 0+} r^{-\gamma} \mu_Z(B(Z(t),r))=\liminf_{r\in \QQ,\, r\to 0+} r^{-\gamma} \mu_Z(B(Z(t),r))\]
is a random variable, so Fatou's lemma implies that
\[ 
\E \left(\liminf_{r\to 0+} r^{-\gamma} \mu_Z(B(Z(t),r))\right)\leq \liminf_{r\to 0+} r^{-\gamma} \E(\mu_Z (B(Z(t),r)))=0.
\]
Therefore, almost surely, we have 
\begin{equation} \label{eq:Fub} \liminf_{r\to 0+} r^{-\gamma} \mu_Z(B(Z(t),r))=0. \end{equation}  Lemma~\ref{l:pr} implies that the set
\[\iD=\left\{(\omega,t)\in \Omega \times \R^n: \liminf_{r\to 0+} r^{-\gamma} \mu(\{s: Z(s,\omega) \in B(Z(t,\omega),r)\})=0\right\}\]
is in $\iF\otimes \iB(\R^n)$. Thus $\iD$ 
is measurable with respect to $\P\times \mu$, so \eqref{eq:Fub} and Fubini's theorem yield that, 
almost surely, for $\mu$ almost every $t\in \R^n$ we
have 
\[\liminf_{r\to 0+} r^{-\gamma} \mu_Z(B(Z(t),r))=0.\]
Hence $\Dim \mu_Z\geq \gamma$ almost surely by Theorem~\ref{t:pmu}, so \eqref{eq:muz} holds. 

Finally, we prove that, almost surely,  
\begin{equation} \label{eq:ZA}
\Dim Z(A)\geq \Dim_Z A.
\end{equation} 
Assume that $\mu\in \iP_c(A)$ is arbitrarily fixed. As $Z$ is jointly measurable, all its sample paths are Borel measurable. Thus Lemma~\ref{l:h} and \eqref{eq:muz} imply that, almost surely,  
\begin{equation} \label{eq:za} 
\Dim Z(A)\geq \Dim \mu_Z \geq \Dim_Z \mu. 
\end{equation}	
By definition we have $\Dim_Z A=\sup \{\Dim_Z \mu: \mu\in \iP_c(A)\}$, so \eqref{eq:za} implies \eqref{eq:ZA}. The proof is complete.
\end{proof}

\subsection{The upper bound} \label{ss:ub}

The aim of this subsection is to prove the following. 

\begin{theorem} \label{t:ub} Let $X=\{(X_1(t),\dots,X_d(t)): t\in \R^n\}$ be a Gaussian random field in $\R^d$ such that $X_1,\dots,X_d$ are independent, centered Gaussian random fields with almost surely continuous sample paths. Let $f\colon \R^n\to \R^d$ be a Borel map, and let $A\subset \R^n$ be an analytic set on which $X$ is regular. If $Z=X+f$ or $Z=(X+f)^*$ then, almost surely,
\begin{equation*} \Dim Z(A)\leq \Dim_Z A.
\end{equation*} 	
\end{theorem}

Theorems~\ref{t:lbb} and \ref{t:ub} will clearly imply our Main Theorem. First we need some preparation. Let $X$ be a Gaussian random field according to Theorem~\ref{t:ub}. For all $1\leq i\leq d$ let $\rho_i$ be the canonical pseudometric for $X_i$, that is, for all $t,s\in \R^n$ let 
\[\rho_i(t,s)=\sqrt{\E (X_i(t)-X_i(s))^2}.\]
Let $f\colon \R^n\to \R^d$ be a Borel map and let $A\subset \R^n$ be an analytic set on which $X$ is regular. For $\mu\in \iP_c(A)$ we will use the notation 
\begin{equation*} 
\Dim_{X,f} \mu=\Dim_{(X+f)} \mu \quad \textrm{and} \quad    \Dim_{X,f}^{*} \mu=\Dim_{(X+f)^{*}} \mu
\end{equation*}
throughout the subsection. Endow $\R^d$ and $\R^{n+d}$ with the maximum norm $||\cdot||$. 
Recall that $N$ denotes a standard normal random variable.
\begin{definition} Since $X_i(t)-X_i(s)$ has distribution $\rho_i(t,s)N$ and $X_1,\dots,X_d$ are independent, for all $0<\beta\leq 1$ and for all $t,s\in \R^n$ and $r>0$ we can define 
\begin{align} \label{eq:F}
\begin{split}
H^{f}_{X}(t,s,r)&=\P(||(X+f)(t)-(X+f)(s)||\leq r) \\
&=\prod_{i=1}^d \P\left(\rho_i(t,s)N\in B(f_i(t)-f_i(s), r)\right),
\end{split} 
\end{align}
where $f_i$ denotes the $i$th coordinate function of $f$.
\end{definition}

Lemma~\ref{l:1} and Remark~\ref{r:1} yield the following corollary.

\begin{corollary} \label{c:1} We have
\begin{align*}\Dim_{X,f} \mu&=\sup\left\{ \gamma: \liminf_{r\to 0+} r^{-\gamma} \int_{\R^n} H^{f}_{X} (t,s,r) \, \mathrm{d} \mu(s)=0 \textrm{ for $\mu$-a.e.}\ t\in \R^n\right\}, \\
\Dim^*_{X,f} \mu&=\sup\left\{ \gamma: \liminf_{r\to 0+} r^{-\gamma} \int_{D(t,r)} H^{f}_{X} (t,s,r) \, \mathrm{d} \mu(s)=0 \textrm{ for $\mu$-a.e.}\ t\in \R^n \right\}.
\end{align*}
\end{corollary}

\begin{lemma} \label{l:gf} Let 
$g\colon \R^n \to \R^d$ be a map such that for each $0<\beta<1$ there is a $C=C(\beta)\in \R^+$ such that for all $1\leq i\leq d$ and $t,s\in A$ we have 
\begin{equation*} 
|g_i(t)-g_i(s)|\leq C\rho_i(t,s)^{\beta} 
\end{equation*} 
Then for all $\mu\in \iP_c(A)$ we have 
\[\Dim \mu_{(g+f)}\leq \Dim_{X,f} \mu \quad \textrm{and} \quad \Dim \mu_{(g+f)^{*}}\leq \Dim^{*}_{X,f} \mu.\]
\end{lemma}

\begin{proof}
Let $\mu\in \iP_c(A)$ and fix an arbitrary $0<\beta<1$. It is enough to prove that   
\begin{equation} \label{eq:gf}
\beta \Dim \mu_{(g+f)}\leq \Dim_{X,f} \mu \quad \textrm{and} \quad \beta \Dim \mu_{(g+f)^{*}}\leq \Dim^{*}_{X,f} \mu.
\end{equation}  	
Corollary~\ref{c:gdf} and scaling imply that $h=g+f$ satisfies 
\begin{align*}  
\beta \Dim \mu_{h}&=\sup\left\{\gamma: \liminf_{r\to 0+} r^{-\gamma} \int_{\R^n} I^{h}_{d}(t,s,r^{\beta}) \, \mathrm{d} \mu(s)=0 \textrm{ for $\mu$-a.e.}\ t\right\}, \\
\beta \Dim \mu_{h^*}&=\sup\left\{\gamma: \liminf_{r\to 0+} r^{-\gamma} \int_{D(t,r^{\beta})} I^{h}_{d}(t,s,r^{\beta}) \, \mathrm{d} \mu(s)=0 \textrm{ for $\mu$-a.e.}\ t \right\}.
\end{align*} 
Let $c=2^{-(1-\beta)^{-1}}\in (0,1)$ and assume that $r\in (0,c)$. Corollary~\ref{c:1}, the above formulas, and $r<r^{\beta}$ imply that in order to prove \eqref{eq:gf} it is enough to show that for all $t,s\in A$ and $r\in (0,c)$ we have 
\[H^{f}_{X} (t,s,r) \lesssim_{d,C} I_{d}^{g+f}(t,s,r^{\beta}).\]
Fix $i\in \{1,\dots,d\}$, it is enough to prove that
\[\P(\rho_i(t,s)N \in B(f_i(t)-f_i(s),r))\lesssim_{C} \min\left\{ 1,r^{\beta}|(f+g)_i(t)-(f+g)_i(s)|^{-1}\right\}.\]
For the sake of notational simplicity let $\rho=\rho_i(t,s)$ and $a=|f_i(t)-f_i(s)|$. By the symmetry of $N$ it is enough to prove that 
\begin{equation} \label{eq:ar} \P(\rho N\in B(a,r))\lesssim r^{\beta}(a+\rho^{\beta})^{-1}.
\end{equation} 
First assume that $\rho=0$. If $a\leq r$ then the right hand side of the above equation is larger than $1$, while if $a>r$ then the left hand side is $0$. Thus \eqref{eq:ar} holds with $\leq $ instead of $\lesssim$.  

Now we may assume that $\rho\neq 0$. We will consider four cases. 

\noindent \textbf{Case I.} If $\rho^{\beta}\leq a\leq r^{\beta}$ then
\[r^{\beta}|a+\rho^{\beta}|^{-1}
\geq r^{\beta}/(2a)\geq 1/2,\]
and we are done since the left hand side of 
\eqref{eq:ar} is at most one.  

\bigskip

\noindent \textbf{Case II.} If $a\leq \rho^{\beta}$ and $\rho \leq r$ then 
\[r^{\beta}|a+\rho^{\beta}|^{-1}
\geq r^{\beta}/(2\rho^{\beta})\geq 1/2,\]
and we are done as above. 

\bigskip 

\noindent \textbf{Case III.}
If $a\leq \rho^{\beta}$ and $r \leq \rho$ then using that $\varphi(z)\leq 1$ we have 
\begin{align*}
\P(\rho N\in B(a,r))&=\int_{B(a/\rho,r/\rho)} \varphi(z) \, \mathrm{d} z\\
&\leq 2(r/\rho)\leq 2(r/\rho)^{\beta} \\
&\leq 4 r^{\beta}(a+\rho^{\beta})^{-1},
\end{align*} 
and \eqref{eq:ar} follows.

\bigskip 

\noindent \textbf{Case IV.} Assume that $a\geq \max\{ \rho^{\beta}, r^{\beta}\}$. Then 
using $\varphi(z)\lesssim |z|^{-1}$, that $\varphi$ is decreasing on $[0,\infty)$, and $r<c$ we obtain that  
\begin{align*}
\P(\rho N\in B(a,r))&=\int_{B(a/\rho,r/\rho)} \varphi(z) \, \mathrm{d} z\\
&\leq 2(r/\rho) \varphi((a-r)/\rho)\\
&\lesssim r/(a-r) \\
&\leq r^{\beta}/a \\
&\leq  2r^{\beta}(a+\rho^{\beta})^{-1},
\end{align*} 
hence \eqref{eq:ar} holds. The proof is complete.
\end{proof}

For the following lemma  see the proof of \cite[Theorem~1.3.5]{AT}, where it is enough to assume that $(T,\rho)$ is totally bounded. The original source is Dudley~\cite{D}.

\begin{lemma} \label{l:Crho} 
Let $\{G(t): t\in T\}$ be a centered real-valued Gaussian random field with almost surely continuous sample paths defined on a bounded set $T\subset \R^n$. Assume that there exist constants $\alpha, u \in \R^+$ such that the canonical pseudometric $\rho$ satisfies $N_{\rho}(T,r)\leq \exp(u r^{-\alpha})$ for all $r>0$. Then, almost surely, there is a finite constant $C=C(\alpha, u)$ such that for all $t,s\in T$ we have 
\begin{equation*}  |G(t)-G(s)|\leq C\rho(t,s)^{1-\alpha}.
\end{equation*}
\end{lemma}

Now we are ready to prove Theorem~\ref{t:ub}. 

\begin{proof}[Proof of Theorem~\ref{t:ub}] 
Fix an arbitrary $0<\beta<1$. First assume that $A$ is bounded and for all $N\in \N^+$ there exists a $c\in \R^+$ such that for all $1\leq i\leq d$ we have 
\begin{equation} \label{eq:hhh} 
\rho_i(t,s)\leq c\log^{-N} (1/|t-s|) \quad \textrm{for all} \quad t,s\in A.
\end{equation}
Fix $N\in \N^+$ with $1/N<1-\beta$. Easy computation shows that there is a constant $u=u(A,n, N, c)$ such that for all $i$ and $r>0$ we have 
\[N_{\rho_i}(A,r)\leq \exp(ur^{-1/N})\leq \exp(ur^{\beta-1}).\] 
Applying Lemma~\ref{l:Crho} for $X_1,\dots,X_d$ implies that we can fix an almost sure sample path $g=X$ and a constant $C=C(g,u,\beta)\in \R^+$ such that for all $1\leq i\leq d$ and $t,s\in A$ we have 
\[|g_i(t)-g_i(s)|\leq  C\rho_i(t,s)^{\beta},\] 
so the condition of Lemma~\ref{l:gf} holds. Let $h=g+f$.  It is enough to prove that 
\begin{equation} \label{eq:Dh}
\Dim h(A)\leq \Dim_{(X+f)} A \quad \textrm{and} \quad \Dim h^{*}(A)\leq \Dim_{(X+f)^{*}} A. 
\end{equation} 	
As $h$ and $h^{*}$ are Borel maps, Lemma~\ref{l:h}, Lemma~\ref{l:gf}, and Definition~\ref{d:Z} yield that 
\begin{align*}
\Dim h(A)&=\sup\{\Dim \mu_h: \mu\in \iP_c(A)\}\leq \Dim_{X,f} \mu \leq \Dim_{(X+f)} A, \\ 
\Dim h^{*}(A)&=\sup\{\Dim \mu_{h^*}: \mu\in \iP_c(A)\}\leq \Dim^*_{X,f} \mu \leq \Dim_{(X+f)^{*}} A.
\end{align*}
Hence \eqref{eq:Dh} holds. For the general case assume that $A= \bigcup_{k=1}^{\infty} A_k$ and for all $k,N\in \N^+$ there is a $c\in \R^+$ such that for all $1\leq i\leq d$ we have 
\begin{equation*}
\rho_i(t,s)\leq c\log^{-N}\left(1/|t-s|\right)\quad \textrm{for all} \quad  t,s\in A_k. 
\end{equation*}  
We may assume that all $A_k$ are bounded, otherwise we can write them as a union of countably many bounded sets. Then our initial assumption holds for each $A_k$. Thus the countable stability of packing dimension, our theorem for the sets $A_k$, and Definition~\ref{d:Z} imply that, almost surely,
\[\Dim Z(A)=\sup_{k} \Dim Z(A_k) \leq \sup_{k} \Dim_Z A_k\leq \Dim_Z A.\]
The proof is complete. 
\end{proof}

\section{The proof of Theorem~\ref{t:Gd}} \label{s:tech} 

Let us recall Definition~\ref{d:gdm} and Theorem~\ref{t:Gd}, 

\begin{dddf}
Let $n\in \N$ and let $d\in \N^+$.
Let $\mu$ be a Borel probability measure on $\R^{n+d}$. For $x=(u,v)\in \R^n\times \R^d$ and $r>0$ let $\mu_{u,r}$ denote the Borel measure on $\R^d$ such that for every Borel set $E\subset \R^d$ we have
\begin{equation*} \mu_{u,r}(E)=\mu(D(u,r)\times E).
\end{equation*}
	Let us define
	\begin{equation*}
	G_{d}^{\mu}(x,r)=\int_{\R^d} I_d\left( \frac{y-v}{r}\right) \, \mathrm{d} \mu_{u,r}(y).
	\end{equation*}
\end{dddf} 

\begin{mmmt}  
	Let $\mu\in \iP(\R^{n+d})$ for some $n\in \N$ and $d\in \N^+$. Then
	\[ \Dim \mu=\sup\left\{\gamma: \liminf_{r\to 0+} r^{-\gamma} G_{d}^{\mu}(x,r)=0 \textrm{ for $\mu$-a.e.}\ x\in \R^{n+d}\right\}.
	\]
\end{mmmt}

Before proving Theorem~\ref{t:Gd} we need some preparation.

\begin{notation} 
For all $x=(x_1,\dots,x_d)\in \R^d$ and $h_1,\dots,h_d\geq 0$ define 
\[D(x,h_1,\dots,h_d)=\prod_{i=1}^{d} [x_i-h_i,x_i+h_i].\]
Recall that if $h_i=h$ for all $1\leq i\leq d$ then $D(x,h_1,\dots,h_d)=D(x,h)$.
\end{notation}
The following lemma generalizes \cite[Lemma~2.1]{FM} and admits a similar proof.
\begin{lemma} \label{l:app1}
Let $\nu\in \iP(\R^{d})$. Assume that $r>0$ and $\lambda_1,\dots,\lambda_d, M\geq 1$.
Then
$$\nu(\{x: \nu(D(x,\lambda_1 r,\dots, \lambda_d r))\geq M \nu(D(x,r))\})\leq 4^{d} M^{-1} \lambda_1\cdots \lambda_d.$$
\end{lemma}

\begin{proof} 
Define 
$$A=\{x\in \R^{d}: \nu(D(x,\lambda_1 r,\dots,  \lambda_d r))\geq M \nu(D(x,r))\}.$$
Assume that $x\in \R^{d}$ such that $D(x,r/2)\cap A\neq \emptyset$. If $y\in A\cap D(x,r/2)$ then 
$D(x,r/2)\subset D(y,r)$ and $D(y,\lambda_1 r,\dots, \lambda_d r)\subset D(x,2\lambda_1 r,\dots, 2\lambda_d r)$. 
Thus 
\begin{align} 
\begin{split}
\label{eq:byrr}
\nu(A\cap D(x,r/2))&\leq \nu(D(y,r)) \\
&\leq M^{-1} \nu(D(y,\lambda_1 r,\dots,\lambda_d r) \\
&\leq M^{-1} \nu(D(x,2\lambda_1 r,\dots, 2\lambda_d r)).
\end{split}
\end{align}
Let $\iL^{d}$ denote the Lebesgue measure on $\R^d$. Applying \eqref{eq:byrr} and Fubini's theorem twice implies that 
\begin{align*} \nu(A)&=\int_{A} 1  \, \mathrm{d} \nu(z) \\
&=(r/2)^{-d} \int_{A} \iL^{d} (D(z,r/2))  \, \mathrm{d}  \nu(z) \\
&=(r/2)^{-d} \int_{A} \int_{\R^{d}}  \mathbbm{1}_{D(z,r/2)}(x) \, \mathrm{d} \iL^{d}(x)  
\, \mathrm{d}  \nu(z) \\
&=(r/2)^{-d} \int_{\R^{d}} \nu(A\cap D(x,r/2))  \, \mathrm{d}  \iL^{d}(x) \\
&\leq (r/2)^{-d} M^{-1} \int_{\R^{d}}  \nu(D(x,2\lambda_1 r,\dots, 2\lambda_d r))  \, \mathrm{d}  \iL^{d}(x) \\
&=(r/2)^{-d} M^{-1} \int_{\R^{d}} \iL^{d} (D(x,2\lambda_1 r,\dots, 2\lambda_d r))  \, \mathrm{d}  \nu(x) \\
&=4^{d} M^{-1} \lambda_1 \cdots \lambda_d.
\end{align*}
This completes the proof. 
\end{proof}

The following lemma slightly generalizes  \cite[Lemma~2.2]{FM}.

\begin{lemma} \label{l:app2}
Let $0<a<1$ and let $\eps>0$. There exists a constant $c_0=c_0(a,\eps,d)$ such that for every Borel measure $\nu$ on $\R^d$ for all $0<r_0\leq 1/2$ we have 
\begin{align*} \nu (\{&x: \nu(D(x,h_1,\dots,h_d))>(4h_1/r)^{1+\eps}\cdots (4h_d/r)^{1+\eps} \nu(D(x,r)) \textrm{ for some} \\
&\textrm{$r$ and $h_i$ such that } 0<r<r_0 \textrm{ and } h_i\geq r^a \textrm{ for all i} \})\leq c_0 r_0^{d\eps (1-a)}.
\end{align*}
\end{lemma}

\begin{proof}  
Let $p\in \N$ and $q_i\in \Z$ be arbitrary such that $q_i\leq p$ for all $1\leq i\leq d$. 
Define 
 \[ A_{p,q_1,\dots,q_d}=\{x: \nu(D(x,2^{-q_1},\dots, 2^{-q_d}))\geq 2^{(1+\eps)\sum_{i=1}^{d} (p-q_i)} 
 \nu(D(x, 2^{-p}))\}.\] 
Lemma~\ref{l:app1} implies that  
\begin{equation} \label{eq:muA}
 \nu(A_{p,q_1,\dots,q_d})\lesssim_{d} 2^{-\eps\sum_{i=1}^d(p-q_i)}.
\end{equation} 
Fix $r_0\leq 1/2$ and $p_0\geq 0$ such that 
$r_0=2^{-1-p_0}$. Define 
\begin{align*} 
B=\{&x: \nu(D(x,h_1,\dots, h_d))>(4h_1/r)^{1+\eps}\cdots (4h_d/r)^{1+\eps}\nu(D(x,r))   \\
&\textrm{for some $r$ and $h_i$ such that $0<r<r_0$ and $h_i\geq r^a$ for all $i$} \}.
\end{align*}
First we prove that 
\begin{equation} \label{D_0}
B\subset \bigcup_{p\geq p_0} \bigcup_{q_1\leq ap} \dots \bigcup_{q_d\leq ap} A_{p,q_1,\dots,q_d}.
\end{equation}
Indeed, assume that $x\in B$. 
Let $0<r<r_0$ and $h_i\geq r^a$ such that 
\begin{equation} \label{eq:xdef}
\nu(D(x,h_1,\dots,h_d))> (4h_1/r)^{1+\eps}\cdots (4h_d/r)^{1+\eps}\nu(D(x,r)).
\end{equation}
There are unique numbers $p\in \N$ and $q_i\in \Z$ such that  $2^{-p} \leq r < 2^{-p+1}$ and $2^{-q_i-1}<h_i\leq 2^{-q_i}$ for all $1\leq i\leq d$. Then $r<r_0$ implies that $2^{-p}<2^{-1-p_0}$, so $p>p_0$. The inequalities $h_i\geq r^a$ yield that $2^{-q_i} \geq 2^{-pa}$, that is, $q_i\leq ap$. Therefore, in order to show \eqref{D_0} it is enough to prove that $x\in A_{p,q_1,\dots,q_d}$. The above inequalities for $r$ and $h_i$ and \eqref{eq:xdef} imply that 
\[\nu(D(x,2^{-q_1},\dots,2^{-q_d}))> 2^{(1+\eps)\sum_{i=1}^{d} (p-q_i)} \nu(D(x,2^{-p})),\]
hence $x\in A_{p,q_1,\dots,q_d}$, so \eqref{D_0} holds.
Finally, \eqref{D_0}, \eqref{eq:muA}, and elementary calculations show that
\begin{align*} \nu(B)&\leq \nu \left(\bigcup_{p\geq p_0} \bigcup_{q_1\leq ap} \dots \bigcup_{q_d\leq ap} A_{p,q_1,\dots,q_d}\right) \\
&\lesssim_{d} \sum_{p\geq p_0} \sum_{q_1\leq ap}\dots 
\sum_{q_d\leq ap} 2^{-\sum_{i=1}^d(p-q_i)\eps}\\
	&=\sum_{p\geq p_0} (2^{\eps p(a-1)}/(1-2^{-\eps}))^d \\
	&\lesssim_{\eps,d} \sum_{p\geq p_0} 2^{\eps d p (a-1)} \\
	&\lesssim_{\eps,d,a} (2^{-p_0-1})^{d\eps (1-a)}\\
	&=r_0^{d\eps (1-a)}.
	\end{align*}
The proof is complete.
\end{proof}

The Borel-Cantelli Lemma implies the following. 

\begin{corollary} \label{c:main} Let $0<a<1$ and let $\eps>0$. For every Borel measure $\nu$ on $\R^d$ for $\nu$ almost every $x$ there is an $r_0=r_0(x)\in (0,1)$ such that for all $r,h_i$ satisfying $0<r<r_0$ and $h_i\geq r^a$ for all $1\leq i\leq d$ we have 
\[
\nu(D(x,h_1,\dots,h_d))\leq \nu(D(x,r)) \prod_{i=1}^d (4h_i/r)^{1+\eps}.\]
\end{corollary} 
 
Integration by parts formula for Riemann-Stieltjes integrals in $\R^d$ was obtained by Young~\cite[Formula~(18)]{Y}, for a slightly more general version see Zaremba~\cite[Proposition~2]{Za}. Applying the latter one twice and taking limits yields the following.

\begin{theorem} \label{t:parts} Let $\mu$ be a finite Borel measure on $\R^d$. Define $g\colon [0,\infty)^d\to [0,\infty)$ as $g(x)=\mu\left(\prod_{i=1}^d [0,x_i)\right)$. 
Let $f\colon [0,\infty)^d\to \R$ be a continuous function such that $f(x)\to 0$ as $|x|\to \infty$. Then 
\[\int_{[0,\infty)^d} f(x)  \, \mathrm{d} \mu(x)=(-1)^d \int_{[0,\infty)^d} g(x)  \, \mathrm{d} f(x).\]	
\end{theorem} 

Now we are able to prove Theorem~\ref{t:Gd}.

\begin{proof}[Proof of Theorem~\ref{t:Gd}] Since $G_d^{\mu}(x,r)\geq \mu(D(x,r))$ for all $x\in \R^{n+d}$ and $r>0$, Theorem~\ref{t:pmu} implies that 
	\[ \Dim \mu\geq \sup\left\{\gamma: \liminf_{r\to 0+} r^{-\gamma} G_{d}^{\mu}(x,r)=0 \textrm{ for $\mu$-a.e.}\ x\in \R^{n+d}\right\}.\]
	For the other direction let $0\leq \alpha<\beta$ be arbitrary and let $\delta=(\beta-\alpha)/3$. It is enough to prove that for $\mu$-almost every $x\in \R^{n+d}$ if 
\begin{equation}  \label{liminf1}
	\liminf_{r\to 0+} r^{-\beta} \mu(D(x,r))=0,
\end{equation}
	then 
\begin{equation} \label{liminf2}
	\liminf_{r\to 0+} r^{-\alpha} G_d^{\mu}(x,r)=0.
\end{equation} 
Fix $\eps>0$ and $0<a<1$ such that 
\begin{equation} \label{eq:eps}	
\max\{(1-a)(1+\eps)d,\eps \beta(d+1)\}<\delta.
\end{equation}

Fix $x_0\in \R^{n+d}$ for which \eqref{liminf1} holds and Corollary~\ref{c:main} is satisfied with $\eps$, $a$, and $r_0=r_0(x_0)\in (0,1)$. We may assume that $x_0=\mathbf{0}=(\mathbf{0}_n,\mathbf{0}_d)$ and for the sake of notational simplicity let $D(\mathbf{0},r)=D(r)$. Fix $r\in (0,r_0)$ such that $\mu(D(r))<r^{\beta}$. Let 
\begin{align*} 
s&=r^{d}(\mu(D(r)))^{-1/(1+\eps)}, \\
t&=\log s-d\log r=\frac{-\log(\mu(D(r)))}{1+\eps}.
\end{align*} 
Let $\nu=\mu_{\mathbf{0}_n,r}$ be the measure according to Definition~\ref{d:gdm}.
Let $g\colon [0,\infty)^d\to [0,\infty)$ be defined by
	\[ g(h_1,\dots,h_d)=\nu\left(\prod_{i=1}^d [0,h_i)\right)=\mu \left([-r,r]^d\times \prod_{i=1}^d [0,h_i)\right).\]
	By Corollary~\ref{c:main} and \eqref{eq:eps} for all $h_1,\dots,h_d\geq r$ we have 
\begin{align} \label{bxh}
\begin{split}  
g(h_1,\dots,h_d)&\leq g(\max\{h_1,r^a\},\dots, \max\{h_d,r^a\}) \\
&\leq \nu\left([-r,r]^d\right) \prod_{i=1}^{d} (4\max\{h_i,r^a\}/r)^{1+\eps} \\ 
&\leq \nu\left([-r,r]^d\right) (4r^a/r)^{(1+\eps)d}
\prod_{i=1}^{d}(4h_i/r)^{(1+\eps)}\\
&\lesssim_d \mu(D(r)) r^{-\delta} \prod_{i=1}^{d}(h_i/r)^{(1+\eps)}.
\end{split}
\end{align}

Let 
\[A=[r,\infty)^d\cap\{(h_1,\dots,h_d): h_1\cdots h_d\leq s\} \quad \textrm{and} \quad B=[r,\infty)^d\setminus A.\]

By Theorem~\ref{t:parts} we have 
\begin{align*}  \int_{[0,\infty)^d} I_d(y/r)  \, \mathrm{d} \nu(y)&=(-1)^d \int_{[0,\infty)^d} g(h) \, \mathrm{d} I_d(h/r) \\
&=(-1)^d \int_{[r,\infty)^d} g(h) \frac{\partial^{d} I_d(h/r)}{\partial h_1\cdots \partial h_d}  \, \mathrm{d} h \\
&=\int_{[r,\infty)^d} g(h) r^d(h_1\cdots h_d)^{-2} \, \mathrm{d} h  \\
&=a(r)+b(r),
\end{align*} 
where 
\[a(r)=\int_{A} g(h) r^d(h_1\cdots h_d)^{-2} \, \mathrm{d} h  \quad \textrm{and} \quad 
b(r)=\int_{B} g(h) r^d(h_1\cdots h_d)^{-2} \, \mathrm{d} h. \]
In order to prove \eqref{liminf2} it is enough to show that 
\begin{equation} \label{eq:ab} 
a(r)=o(r^{\alpha}) \quad \textrm{and} \quad b(r)=o(r^{\alpha}) \quad \textrm{as } r\to 0+.
\end{equation} 
 Indeed, dividing $\R^d$ into $2^d$ parts according to the signs of coordinates similarly yields  
\[G_d^{\mu}(\mathbf{0},r)=\int_{\R^d} I_d(y/r)  \, \mathrm{d} \nu(y)\lesssim_d a(r)+b(r)=o(r^{\alpha}) \]
as $r\to 0+$, and \eqref{liminf2} will follow. 

Now we will prove \eqref{eq:ab}. Let $\iL^{d-1}$ denote the Lebesgue measure on the affine hyperplanes of $\R^d$. 
The substitution $h_i=re^{y_i}$, Fubini's theorem, and the inequality $\log(1/v)\lesssim_{\eps} v^{-\eps}$ yield that 
\begin{align*} r^{-\eps d} \int_{A}\prod_{i=1}^{d} h_i^{\eps-1} \, \mathrm{d} h &=\int_{y_i\geq 0,~\sum_i y_i\leq t} 
e^{\eps(y_1+\dots+y_d)} \, \mathrm{d} y \\
&\leq  \int_{u=0}^{t} e^{\eps u} \iL^{d-1}\left(\{y_i: y_i\geq 0,\, y_1+\dots+y_d=u\}\right) \, \mathrm{d} u \\
&\lesssim_d  \int_0^{t} e^{\eps u} u^{d-1} \, \mathrm{d} u \\ 
&\leq e^{\eps t}  t^d \\
&\lesssim_{\eps} \mu(D(r))^{-\eps/(1+\eps)} \mu(D(r))^{-d\eps/(1+\eps)} \\
&\leq \mu(D(r))^{-\eps(d+1)}.
\end{align*} 
Inequalities~\eqref{bxh}, \eqref{eq:eps}, and $\mu(D(r))<r^{\beta}$ imply that  
\begin{align} \label{ar} 
\begin{split}
a(r) &\lesssim_d \mu(D(r)) r^{-\delta-\eps d} \int_{A}\prod_{i=1}^{d} h_i^{\eps-1} \, \mathrm{d} h \\
&\lesssim_{d,\eps} \mu(D(r))^{1-\eps(d+1)} r^{-\delta}\\
&\leq r^{\beta-\beta \eps (d+1)-\delta} \\
&\leq r^{\beta-2\delta}=o(r^{\alpha})
\end{split} 
\end{align}	
as $r\to 0+$. 
Finally, $g(h)\leq 1$, the above steps, and \eqref{eq:eps} imply that 

\begin{align} \label{br} 
\begin{split} 
b(r)&\leq r^d \int_{B} (h_1\cdots h_d)^{-2}  \, \mathrm{d} h
\\&\leq \int_{y_i\geq 0,~\sum_i y_i\geq t} 
e^{-(y_1+\dots+y_d)} \, \mathrm{d} y \\
&= \int_{u=t}^{\infty} e^{-u} \iL^{d-1}\left(\{y_i: y_i\geq 0,\, y_1+\dots+y_d=u\}\right) \, \mathrm{d} u \\
&\lesssim_d  \int_{u=t}^{\infty} e^{-u} u^{d-1}  \, \mathrm{d} u \\
&\lesssim_d  e^{-t}(t^{d-1}+1) \\
&\lesssim_{\eps} \mu(D(r))^{1/(1+\eps)} \mu(D(r))^{-\eps(d-1)/(1+\eps)} \\
&\leq \mu(D(r))^{1-\eps d} \\
& \leq r^{\beta-\beta \eps d}\\
&\leq r^{\beta-\delta}=o(r^{\alpha})
\end{split} 
\end{align}   
as $r\to 0+$. Inequalities \eqref{ar} and \eqref{br} imply \eqref{eq:ab}, and the proof is complete.  
\end{proof} 

\section{Packing dimension of the graph of fractional Brownian motion} \label{s:graph} 

The goal of this section is to prove Theorems~\ref{t:ad} and \ref{t:ad2}. 

For $x\in \R$ we will denote by $\lfloor x \rfloor$ and $\lceil x \rceil$ the lower and upper integer part of $x$, respectively.
The following lemma is basically \cite[Lemma~4.1]{TX}.

\begin{lemma} \label{l:Eg}
Let $A\subset \R$ be a compact set and let $\gamma<\Dim A$. Let $\theta>0$. Then there is a compact set $E_{\gamma}=E(\gamma,\theta)\subset A$ with 
$\Dim E_{\gamma}\geq \gamma$ such that $E_{\gamma}=\bigcap_{n=1}^{\infty} D_n$, where 
\[D_n=\bigcup_{i_1=1}^{N_0}\dots \bigcup_{i_n=1}^{N_{i_1\dots i_{n-1}}}
I_{i_1\dots i_n}\] 
with closed intervals $I_{i_1\dots i_n}$, and for all $n$ and appropriate indexes
\begin{enumerate}[(i)]
\item the intervals $I_{i_1\dots i_n}$ ($1\leq i_n\leq N_{i_1\dots i_{n-1}})$ are disjoint subintervals of $I_{i_1\dots i_{n-1}}$ for all $n\geq 2$, 
\item \label{i2} there exist $\eta_{i_1\dots i_n}>0$ such that
$\eta_{i_1\dots i_n}^{\theta}<\eta_{i_1\dots i_{n-1}}$ and the  distance between the intervals $I_{i_1\dots i_{n-1}j}$ is at least $\eta_{i_1\dots i_{n-1}}$,
\item \label{i3}
there exists a Borel probability measure $\mu$ such that $\supp(\mu)=E_{\gamma}$ and we have
$\mu(I_{i_1\dots i_n})\leq \eta_{i_1\dots i_{n-1}}^{\gamma}$. 	
\end{enumerate} 
\end{lemma}  

Now we are ready to prove Theorem~\ref{t:ad}. 

\begin{proof}[Proof of Theorem~\ref{t:ad}] Let $\{X(t): t\in \R\}$ be a $d$-dimensional fractional Brownian motion and let $Z=X^{*}$. By Corollary~\ref{c:h2} we may assume that $A$ is compact. As packing dimension cannot increase under a projection, Theorem~\ref{t:TX} implies that, almost surely, we have
\[\Dim Z(A)\geq \Dim X(A)\geq \frac{\beta d}{\alpha d+\beta(1-\alpha d)}.\]  
Let $\theta=1/(d+1-\alpha d)$ and fix $\gamma<\beta=\Dim A$. Let $E_{\gamma}=E(\gamma, \theta)\subset A$ be a compact set and let $\mu$ be a Borel probability measure with $\supp(\mu)=E_{\gamma}$ according to Lemma~\ref{l:Eg}. By Theorem~\ref{t:lbb} it is enough to show that 
\begin{equation} \label{eq:gam}  
\Dim_{Z} E_{\gamma}\geq \gamma(d+1-\alpha d).
\end{equation} 
Fix an arbitrary $t\in E_{\gamma}$ and let $\{i_n\}_{n\geq 1}$ be the sequence of positive integers such that $t\in I_{i_1\dots i_n}$ for all $n\in \N^+$. 
Let us endow $\R^{1+d}$ with the maximum norm $||\cdot||$. By Remark~\ref{r:1} it is enough to prove that for all $n\geq 2$ we have
\begin{equation} \label{eq:i12}
\E (\mu_Z(D(Z(t),\eta_{i_1\dots i_{n-1}}^{\theta})))\lesssim_{\alpha,d} \eta_{i_{1}\dots i_{n-1}}^{\gamma}.
\end{equation}  
Indeed, \eqref{eq:i12} implies that 
for all $\delta<\gamma/\theta$ we have 
\[\liminf_{r\to 0+} r^{-\delta} \E(\mu_Z (D(Z(t),r)))=0, \]
hence 
\[\Dim_Z E_{\gamma}\geq \gamma/\theta=\gamma(d+1-\alpha d).\]
 As $\gamma<\beta$ was arbitrary, this implies \eqref{eq:gam}. 

Finally, we prove \eqref{eq:i12}. 
Fix $n\geq 2$ and let $\eta=\eta_{i_1\dots i_{n-1}}$. By Lemma~\ref{l:Eg}~\eqref{i2} the interval $B(t,\eta^{\theta})$ intersects only $n$th level intervals of the form $I_{i_1\dots i_{n-1}j}$, which are separated from each other by at least $\eta$. Lemma~\ref{l:Eg}~\eqref{i3} yields that every interval $I=I_{i_1\dots i_{n-1} j}$ satisfies $\mu(I)\leq \eta^{\gamma}$. Applying Lemma~\ref{l:1}, the inequality $\P(N\leq r)\lesssim \min\{1,r\}$, the above observations, and $\alpha d<1$ in this order implies 
\begin{align*} 
\E (\mu_Z(D(Z(t),\eta^{\theta})))&=\int_{B(t,\eta^{\theta})} \P(||X(s)-X(t)||\leq \eta^{\theta})  \, \mathrm{d} \mu(s)  
\\
&\lesssim_{d}
\int_{B(t,\eta^{\theta})} \min\left\{1,\frac{\eta^{\theta d}}{|s-t|^{\alpha d}}\right\} 
 \, \mathrm{d} \mu(s) \\
&\leq \mu(I_{i_1\dots i_n})+2\eta^{\theta d} \eta^{\gamma}
\sum_{j\leq \eta^{\theta}/\eta} \frac{1}{(j\eta)^{\alpha d}} \\
&\leq \eta^{\gamma}+2\eta^{\gamma+(\theta-\alpha)d} \sum_{j\leq \eta^{\theta-1}} j^{-\alpha d} \\
&\lesssim_{\alpha d} \eta^{\gamma} +\eta^{\gamma+(\theta-\alpha)d} \eta^{(\theta-1)(1-\alpha d)}\\
&=2 \eta^{\gamma}.
\end{align*} 
Thus \eqref{eq:i12} holds, and the proof is complete. 
\end{proof}    

Before proving Theorem~\ref{t:ad2} we construct the compact sets $A_{\beta}\subset \R$ following \cite{TX}. Fix $\beta\in (0,1)$. By \cite[Lemma~3.1]{TX} we can choose two sequences of positive numbers $\{\delta_k\}_{k\geq 0}$ and $\{\eta_k\}_{k\geq 1}$ with $\delta_0<1/2$ such that for all $k\geq 1$ we have 
\begin{align} 
\delta_k&<\eta_k<1, \label{eq:l1} \\
\delta_{k-1}&=2\eta_k^{1-\beta}, \label{eq:l2} \\ 
m_1 m_2 \cdots m_k&\leq \delta_k^{-2^{-(k+1)}}, \label{eq:l3}  
\end{align} 
where $m_k=\lfloor \eta_k^{-\beta} \rfloor$. 
Let $D_0=[0,1]$ and let $D_1=\bigcup_{i_1=1}^{m_1} I_{i_1}$ be the union of $m_1$ closed subintervals of $[0,1]$ of length $\delta_1$ such that the distance between any two consecutive intervals equals $\eta_1$. As $m_1(\eta_1+\delta_1)<1$ by \eqref{eq:l1} and \eqref{eq:l2}, this is possible. Assume that 
$D_{k-1}=\bigcup_{i_1=1}^{m_1}\dots \bigcup_{i_{k-1}=1}^{m_{k-1}} I_{i_1 \dots i_{k-1}}$ is constructed as the union of $m_1\cdots m_{k-1}$ closed intervals for some $k\geq 2$. Fix $i_1,\dots,i_{k-1}$. 
By \eqref{eq:l1} and \eqref{eq:l2} we have $m_k(\eta_k+\delta_k)\leq \delta_{k-1}$ and the length of $I_{i_1\dots i_{k-1}}$ equals $\delta_{k-1}$, so we can construct closed subintervals $I_{i_1\dots i_k}$ $(1\leq i_k\leq m_k)$ in $I_{i_1\dots i_{k-1}}$ with length $\delta_k$ such that the distance of any two consecutive intervals is exactly $\eta_k$. Define 
\[D_k=\bigcup_{i_1=1}^{m_1}\dots \bigcup_{i_{k}=1}^{m_{k}} I_{i_1 \dots i_{k}},\]
and let 
\[A_{\beta}=\bigcap_{k=1}^{\infty} D_k.\]
Now we are ready to prove Theorem~\ref{t:ad2} 
\begin{proof}[Proof of Theorem~\ref{t:ad2}] 
Fix $0<\beta<1$ and let $A_{\beta}$ be the set defined above. Clearly $A_{\beta}\subset [0,1]$ is compact and by \cite[Lemma~3.2]{TX} we have $\Dim A_{\beta}=\beta$. For the sake of simplicity endow $\R^{1+d}$ with the maximum norm, this does not modify the packing dimension. Let 
$f\colon A_{\beta} \to \R^d$ be a fixed $\alpha$-H\"older continuous function, where $0<\alpha<1$ and $d\in \N^+$. We may assume by scaling 
that for all $s,t\in A_{\beta}$ we have 
\begin{equation} \label{eq:Hold} |f(t)-f(s)|\leq |t-s|^{\alpha}.
\end{equation}
As packing dimension is smaller than equal to upper Minkowski dimension, it is enough to prove that  
\begin{equation} \label{eq:uM} 
\overline{\dim}_M f^{*}(A_{\beta})\leq \max\left\{ \frac{\beta d}{\alpha d+\beta(1-\alpha d)},\beta(d+1-\alpha d)\right\}.
\end{equation} 
Before proving \eqref{eq:uM} we need some technical preparation. 
Define the functions $g,h\colon [1,\infty)\to \R^+$ as 
\[g(x)=\frac{\beta}{\alpha(1-\beta)x},\]
and  
\begin{equation*} 
h(x)=\begin{cases} (1-1/x)d &\textrm{ if } 1\leq x\leq 1/\alpha, \\
(1-\alpha)d+1-1/(\alpha x) &\textrm{ if } x>1/\alpha.
\end{cases} 
\end{equation*}
Clearly $g,h$ are continuous functions such that $g$ is strictly decreasing, and $h$ is strictly increasing. We also have 
$h(1)<g(1)$ and $\lim_{x\to \infty} g(x)<\lim_{x\to \infty} h(x)$, so there is a unique point $x_*\in [1,\infty)$ such that 
$g(x_*)=h(x_*)$. Hence we obtain 
\begin{equation} \label{eq:min} 
\min_{x\geq 1} \max\{g(x),h(x)\}=g(x_*).
\end{equation}  
Calculating $x_*$ separately in the cases $h(1/\alpha)\geq g(1/\alpha)$ and $h(1/\alpha)<g(1/\alpha)$, and comparing the two values of $g(x_*)$ yield that 
\begin{equation} \label{eq:g*} g(x_*)=\max\left\{ \frac{\beta d}{\alpha d+\beta(1-\alpha d)},\beta(d+1-\alpha d)\right\}.
\end{equation} 
Now we are ready to prove \eqref{eq:uM}. Let $0<\eps<\delta_1$ be arbitrary and let $k\geq 2$ such that $\delta_k\leq \eps <\delta_{k-1}$. Let $\eps=\delta_{k-1}^{x}$ for some $x\geq 1$. Inequality \eqref{eq:l3} implies that 
\begin{equation} \label{eq:mpr} 
m_1\cdots m_{k-1}\leq \delta_{k}^{-2^{-k}}\leq \eps^{-2^{-k}}=\eps^{-\alpha o_k(1)},
\end{equation} 
where $o_k(1)\to 0$ as $k\to \infty$. By \eqref{eq:l2} we obtain that 
\begin{equation} \label{eq:mpr2} 
m_k\leq \eta_k^{-\beta}\lesssim_{\beta} \delta_{k-1}^{-\beta/(1-\beta)}=\eps^{-\beta/((1-\beta)x)}.
\end{equation} 
Since $A_{\beta}\subset D_k$, inequality \eqref{eq:Hold} implies that 
\begin{equation} \label{eq:dk1}
N(f^{*}(A_{\beta}),\eps^{\alpha})\leq N(f^{*}(D_k),\eps^{\alpha})\leq m_1\cdots m_k.  
\end{equation}
 
First assume that $1\leq x\leq 1/\alpha$. Then $\delta_{k-1}\leq \eps^{\alpha}$. Since every $(k-1)$st level interval $I=I_{i_1\dots i_{k-1}}\subset D_{k-1}$ has length $\delta_{k-1}\leq \eps^{\alpha}$, by \eqref{eq:Hold} we have 
\begin{align} 
\begin{split} 
\label{eq:dk2} 
N(f^{*}(I),\eps^{\alpha}) &\leq \lceil (\delta_{k-1}/\eps)^{\alpha} \rceil ^{d}\leq (2(\delta_{k-1}/\eps)^{\alpha})^d \\
&\lesssim_{d} (\delta_{k-1}/\eps)^{\alpha d}=\eps^{-(1-1/x)\alpha d}.
\end{split} 
\end{align} 
Thus \eqref{eq:dk1}, \eqref{eq:dk2}, \eqref{eq:mpr}, and \eqref{eq:mpr2} yield that 
\begin{align}  
\begin{split} 
\label{eq:N1} 
N(f^{*}(A_{\beta}),\eps^{\alpha})& 
\lesssim_{d} \min\left\{m_1\cdots m_k, m_1\cdots m_{k-1} \eps^{-(1-1/x)\alpha d}\right\} \\
&=m_1\cdots m_{k-1} \min\{m_k,  \eps^{-(1-1/x)\alpha d}\} \\ 
&\lesssim_{\beta} \eps^{-\alpha o_k(1)} \min\{ \eps^{-\beta/(x(1-\beta))}, \eps^{-(1-1/x)\alpha d}\} \\
&=\eps^{-\alpha(o_k(1)+\max\{g(x),h(x)\})}.
\end{split}
\end{align}

Now assume that $x>1/\alpha$, so $\delta_{k-1}>\eps^{\alpha}$. If $J$ is any interval of length $\eps^{\alpha}$, then 
\eqref{eq:Hold} implies that 
\[ N(f^{*}(J),\eps^{\alpha})\leq \lceil \eps^{\alpha^2-\alpha} \rceil^{d}\lesssim_{d} \eps^{(\alpha^2-\alpha)d}. \]
As every $(k-1)$st level interval $I=I_{i_1\dots i_{k-1}}\subset D_{k-1}$ of length $\delta_{k-1}$ can be covered by $\lceil \delta_{k-1}/\eps^{\alpha} \rceil\leq 2\delta_{k-1}/\eps^{\alpha}$ intervals of length $\eps^{\alpha}$, we obtain that 
\begin{equation} \label{eq:dk3} 
N(f^{*}(I),\eps^{\alpha}) \lesssim_{d} \delta_{k-1} \eps^{(\alpha^2-\alpha)d-\alpha}=\eps^{1/x+(\alpha^2-\alpha)d-\alpha}.
\end{equation} 
Similarly to \eqref{eq:N1} we obtain that 
\begin{align}  
\begin{split} 
\label{eq:N2} 
N(f^{*}(A_{\beta}),\eps^{\alpha})& 
\lesssim_{d} \min\left\{m_1\cdots m_k, m_1\cdots m_{k-1} \eps^{1/x+(\alpha^2-\alpha)d-\alpha} \right\} \\
&=m_1\cdots m_{k-1} \min\{m_k, \eps^{1/x+(\alpha^2-\alpha)d-\alpha}\} \\ 
&\lesssim_{\beta} 
\eps^{-\alpha o_k(1)} \min\{ \eps^{-\beta/(x(1-\beta))}, \eps^{1/x+(\alpha^2-\alpha)d-\alpha}\} \\
&=\eps^{-\alpha(o_k(1)+\max\{g(x),h(x)\})}.
\end{split}
\end{align} 

Inequalities \eqref{eq:N1} and \eqref{eq:N2} with \eqref{eq:min} and \eqref{eq:g*} imply that 
\[\overline{\dim}_M f^{*}(A_{\beta})\leq \min_{x\geq 1} \max\{g(x),h(x)\}=\max\left\{ \frac{\beta d}{\alpha d+\beta(1-\alpha d)},\beta(d+1-\alpha d)\right\}. \]
Thus \eqref{eq:uM} holds, and the proof is complete.
\end{proof}

\end{document}